\documentclass[a4paper,12pt]{amsart}
\usepackage[utf8]{inputenc}
\usepackage{fullpage}
\usepackage[english]{babel}
\usepackage{enumerate}
\usepackage[usenames,dvipsnames]{xcolor}
\usepackage[colorlinks=true,linkcolor=Blue,citecolor=Green]{hyperref}
\usepackage{bookmark}
\usepackage{amsmath,amsthm,amsfonts,amssymb}
\usepackage{graphicx}
\newcommand\invisiblesection[1]{%
	\refstepcounter{section}%
	\addcontentsline{toc}{section}{\protect\numberline{\thesection}#1}%
	\sectionmark{#1}}

\newtheorem{theorem}{Theorem}[section]
\newtheorem{prop}[theorem]{Proposition}
\newtheorem{claim}{Claim}
\newtheorem{cor}[theorem]{Corollary}
\newtheorem{lemma}[theorem]{Lemma}
\theoremstyle{definition}
\newtheorem{definition}{Definition}

\newtheorem{question}[theorem]{Question}

\newcommand{\Hc}{\mathcal{H}}
\newcommand{\sm}{\frac{|\sigma(\mathcal{H})|}{|\mathcal{H}|}}
\newcommand{\R}{\mathbb{R}}

\author{Attila Jung}
\address{A.J.: Inst. of Mathematics, ELTE E\"otv\"os Lor\'and University, Budapest, Hungary}
\email{jungattila@gmail.com}

\title{Shadow of hypergraphs with bounded degree}

\makeatletter
\@namedef{subjclassname@2020}{%
	\textup{2020} Mathematics Subject Classification}
\makeatother

\begin{document}

\begin{abstract}
	We study the size of the shadow of $k$-uniform hypergraphs with bounded degree. Lower bounds on the ratio of the size of the shadow and the size of the hypergraph are given as a function of the degree bound and $k$. We show that cliques are extremal for a long range of degree bounds, but not for every bound. We give a general, but not sharp lower bound on the shadow ratio and show, that sometimes we can get extremal hypergraphs by deleting disjoint maximal matchings from a clique.
\end{abstract}

\maketitle

\section{Introduction}\label{sec:intro}

A $k$-uniform hypergraph on the underlying set $[n] = \{1, \ldots, n\}$ is a family $\Hc$ of $k$-element subsets of $[n]$. In notation: $\Hc \subset \binom{[n]}{k}$. We call the elements of $\Hc$ \textit{hyperedges}. The \textit{shadow} of $\Hc$ is the $(k-1)$-uniform hypergraph $\sigma(\Hc)$ whose elements are all the $(k-1)$ element subsets of $[n]$ contained in at least one hyperedge of $\Hc$: $\sigma(\Hc) = \{S \in \binom{[n]}{k-1}| \exists H \in \Hc: S \subset H\}$.

The size of $\Hc$ is just the number of hyperedges contained in $\Hc$, denoted by $|\Hc| = |\{H \in \Hc\}|$. The Kruskal-Katona Shadow Theorem gives a sharp lower bound on $|\sigma(\Hc)|$ as a function of $|\Hc|$ \cite{kruskal1963simplices, katona1968theorem}. But what can we say about the size of the shadow, if we know that the maximum degree of $\Hc$ is bounded? Let $\Delta(\Hc) = \max_{v \in [n]} |\{H \in \Hc: v \in H\}|$ be the \textit{maximum degree} of $\Hc$. A \textit{degree bound} of $\Hc$ is any real number $d$ such that $\Delta(\Hc) \leq d$. In this paper, we consider lower bounds on the shadow ratio $\sm$ of $k$-uniform hypergraphs with degree bound $d$ as a function of $k$ and $d$.

In general the shadow ratio can be as close to zero as we want for any fixed $k$, as the sequence of larger and larger cliques show: if $\Hc = \binom{[n]}{k}$, then $\sm = \frac{k}{n-k+1} \to 0$ as $n \to \infty$. But many hypergraph classes with bounded shadow ratio are known. A few examples are $l$-intersecting hypergraphs \cite{katona1964intersection}, hypergraphs with bounded VC-dimension \cite{frankl1984hypergraphs} and hypergraphs with a set of distinct representatives \cite{leck1993minimum}. In all of these cases the largest possible cliques are extremal. What can we say about the family of $k$-uniform hypergraphs with $\Delta(\Hc) \leq d$?

If $k = 2$, $\Hc$ is just a graph, and its shadow is the set of nonisolated vertices. A simple double counting of the vertex-edge incidences shows us the following.

\begin{prop}\label{prop:graphshadow}
	If $\Hc$ is a $2$-uniform hypergraph with degree bound $d\geq1$, then
	\begin{equation*}
		\sm \geq \frac{2}{d}.
	\end{equation*}
\end{prop}

And, as any $d$-regular graph confirms, the bound on the shadow ratio is tight. Extremal hypergraphs, for which the stated inequalities are tight, are of interest in general. As we will see, cliques are extremal in terms of shadow ratio not just among hypergraphs with the same degree bounds as theirs. But what happens, if the degree bound is so low, it excludes every nontrivial clique? This case is rather different. We start with that by stating a stronger result than the one about the shadow ratios.

\begin{theorem}\label{thm:lowdegree}
	If $1 \leq d < k$, $\Hc$ is a $k$-uniform hypergraph with degree bound $d$ and $|\Hc| = qd + r$ with nonnegative integers $r < d$ and $q$, then
	\begin{equation*}
		|\sigma(\Hc)| \geq \left(qd + r\right)k - q\binom{d}{2} - \binom{r}{2}.
	\end{equation*}
\end{theorem}

The tightness of the inequality is confirmed by the disjoint union of $q$ hypergraphs of size $d$, each on $k+1$ vertices and one hypergraph of size $r$, also on $k+1$ vertices. The corresponding result about shadow ratios is an easy corollary.

\begin{cor}\label{cor:lowdegreeratio}
	If $1 \leq d < k$ and $\Hc$ is a $k$-uniform hypergraph with degree bound $d$, then
	\begin{equation*}
		\sm \geq k - \frac{d-1}{2}.
	\end{equation*}
\end{cor}

What happens if $d \geq k$? We first start with a general lower bound on the shadow ratio. It is not tight most of the times but gives a usable lower bound for every degree bound. If $k$ is a positive integer and $x \geq k$ is a real number, we can define the binomial coefficient $\binom{x}{k}$ by $\binom{x}{k} = \frac{x(x-1)\ldots(x-k+1)}{k!}$. Every positive integer can be expressed as such a binomial coefficient with a suitable $x$.

\begin{theorem}\label{thm:lovaszdegree}
	If $x \geq k-1 \geq 2$ is a real number and $\Hc$ is a $k$-uniform hypergraph with degree bound \[d \leq \binom{x}{k-1},\] then
	\begin{equation*}
		\sm \geq \frac{k}{x-k+2}.
	\end{equation*}
\end{theorem}

This bound is tight only if $x$ is an integer. In this case, an extremal graph is a clique on $x+1$ vertices. What is interesting is that these cliques are extremal even among hypergraphs with much larger degree bound.

\begin{theorem}\label{thm:longintervall}
	If $t \geq k \geq 3$ are integers and $\Hc$ is a $k$-uniform hypergraph with degree bound
	\begin{equation*}
		d \leq \binom{t}{k-1}
		+ \binom{t -\left\lceil \frac{(t+1)}{k-1}\right\rceil}{k-2} 
		+ \binom{t -\left\lceil \frac{2(t+1)}{k-1}\right\rceil}{k-3} 
		+ \ldots 
		+ \binom{t - \left\lceil \frac{(k-2)(t+1)}{k-1}\right\rceil}{1},
	\end{equation*}
	then
	\begin{equation*}
		\sm \geq \frac{k}{t-k+2}.
	\end{equation*}
\end{theorem}

What happens after this degree bound? We can show that the clique on $t+1$ vertices is not extremal much longer.

\begin{prop}\label{prop:example}
	For every integer $t \geq k \geq 3$ there exists a $k$-uniform hypergraph $\Hc$ with degree bound $d$ such that
	\begin{equation*}
	d \leq \binom{t}{k-1} 
	+ \binom{t+1 - \left\lceil \frac{t+2}{k} \right\rceil}{k-2},
	\end{equation*}
	and
	\begin{equation*}
	\sm < \frac{k}{t-k+2}.
	\end{equation*}
\end{prop}

The proof of Proposition~\ref{prop:example} is just a construction of a hypergraph with max degree $\binom{t}{k-1} 
+ \binom{t+1 - \left\lceil \frac{t+2}{k} \right\rceil}{k-2}$ and lower shadow ratio, than that of $\binom{[t+1]}{k}$. This hypergraph might be extremal for the given degree bound, but for the proof better lower bounds are needed.

One more interval where we can give extremal hypergraphs, is the case, when the degree bound is smaller than the max degree of $\binom{[t+2]}{k}$, but still close to it. In this case, one can get extremal hypergraphs by deleting some disjoint maximal matchings from $\binom{[t+2]}{k}$ if $k | t+2$.

\begin{theorem}\label{thm:shortintervall}
	If $t \geq k \geq 3$ are integers and $\Hc$ is a $k$-uniform hypergraph with degree bound
	\begin{equation*}
		\binom{t+1}{k-1} 
		- \frac{k-2}{k-1}(t-k+3) 
		\leq d 
		< \binom{t+1}{k-1},
	\end{equation*}
	then
	\begin{equation*}
		\sm 
		\geq \frac{k\binom{t+1}{k-2}}{(k-1)d}
		= \frac{\binom{t+2}{k-1}}{\binom{t+2}{k} - \frac{t+2}{k}\left(\binom{t+1}{k-1} - d\right)}.
	\end{equation*}
\end{theorem}

The paper is organised as follows. In Section~\ref{sec:sizebound} we cover the famous Kruskal-Katona Theorem and give some new bounds on the shadow ratio of bounded size hypergraphs. In Section~\ref{sec:largedegree} we prove Theorem~\ref{thm:lovaszdegree}, \ref{thm:longintervall} and \ref{thm:shortintervall} by connecting the case of bounded degree hypergraphs to the case of bounded size hypergraphs. In Section~\ref{sec:smalldegree} an elementary proof of Theorem~\ref{thm:lowdegree} is given. Section~\ref{sec:example} contains the description of a hypergraph family which proves Proposition~\ref{prop:example}. In the end, in Section~\ref{sec:conclusion} we pose some open questions.

\section{Shadow ratio of hypergraphs with bounded size}\label{sec:sizebound}

For uniform hypergraphs of fixed size, the famous Kruskal-Katona Shadow Theorem determines the exact lower bound on the size of the shadow. To state it, we will need  the following

\begin{lemma}[\cite{katona1968theorem}]\label{lem:kbinom}
	For every positive integer $m$ and $k$ there uniquely exists $a_k > a_{k-1} > \ldots a_l \geq l \geq 1$ such that
	\[
	m = \binom{a_k}{k} + \binom{a_{k-1}}{k-1} + \ldots + \binom{a_l}{l}.
	\]
	We call this sum the $k$-binomial representation of $m$.
\end{lemma}

The $k$-binomial representation of positive integers plays a crucial role in the Shadow Theorem. 

\begin{definition}\label{def:shadowfunction}
	If $m$ is a positive integer with $k$-binomial representation $m = \binom{a_k}{k} + \binom{a_{k-1}}{k-1} + \ldots + \binom{a_l}{l}$ then the value of the $k$-order shadow function at $m$ is
	\[
	F_k(m) =  \binom{a_k}{k-1} + \binom{a_{k-1}}{k-2} + \ldots + \binom{a_l}{l-1}.
	\]
\end{definition}

The $k$-order shadow function determines the exact lower bound on the size of the shadow of $k$-uniform hypergraphs of size $m$. This statement is the Kruskal-Katona Shadow Theorem \cite{kruskal1963simplices, katona1968theorem}.

\begin{theorem}[Kruskal-Katona Shadow Theorem]\label{thm:shadow}
	For every $k$-uniform hypergraph $\Hc$
	\[
	|\sigma(\Hc)| \geq F_k(|\Hc|).
	\]
	Moreover, if $m \leq \binom{a}{k}$, then there exists a $\Hc \subset \binom{[a]}{k}$, $|\Hc| = m$ hypergraph with
	\[
	|\sigma(\Hc)| = F_k(|\Hc|).
	\]
\end{theorem}

The Shadow Theorem determines $\frac{F_k(m)}{m}$ for every value of $k$ and $m$, but these values are hard to handle, even if we just need a statement about hypergraphs of bounded size instead of fixed size. One possible solution is to use a weaker version of the Shadow Theorem due to Lovász \cite{lovasz2007combinatorial}, which gives a not tight, but easily applicable lower bound on the shadow. For any positive integer $k$ and real number $x$, let $\binom{x}{k} = \frac{x(x-1)\ldots(x-k+1)}{k!}$.

\begin{theorem}\label{thm:lovaszshadow}
	If $k$ is a positive integer and $x \geq k$ a real number such that $\binom{x}{k}$ is an integer, then
	\[F_k\left(\binom{x}{k}\right) \geq \binom{x}{k-1}.
	\]
	Equality holds if and only if $x$ is an integer.
\end{theorem}

This lower bound is not tight, but one can easily deduce a result about the shadow ratio of hypergraphs of bounded size from it.

\begin{cor}\label{cor:lovaszshadowratio}
	If $k$ is a positive integer, $x \geq k$ a real number and $m \leq \binom{x}{k}$ a positive integer, then
	\[
	\frac{F_k(m)}{m}
	\geq \frac{\binom{x}{k-1}}{\binom{x}{k}}
	= \frac{k}{x-k+1}.
	\]
\end{cor}

This corollary is also tight if and only if $x$ is an integer. The result about integer $x$s is maybe the first result about shadow ratios due to Sperner \cite{sperner1928satz}.

\begin{lemma}\label{lem:sperner}
	If $k \leq a$ and $m \leq \binom{a}{k}$ are positive integers, then
	\[
	\frac{F_k(m)}{m}
	\geq \frac{k}{a-k+1}.
	\]
\end{lemma}

Corollary \ref{cor:lovaszshadowratio} gives a strictly monotone decreasing bound of $\frac{F_k(m)}{m}$ in $m$. But the values of this shadow ratio function are far from strictly decreasing. We show two generalizations of Lemma~\ref{lem:sperner}, the first is about values of $m$ which are slightly bigger than $\binom{a}{k}$, the second is about values that are slightly smaller.

\begin{lemma}\label{lem:longintervall}
	Let $a \geq k > 0$ be positive integers. If
	\[
	0 < m \leq \sum_{u=0}^{k-1}\binom{a - \left\lceil\frac{u(a+1)}{k}\right\rceil}{k-u},
	\]
	then
	\[
	\frac{F_k(m)}{m}
	\geq \frac{k}{a-k+1}.
	\]
\end{lemma}

\begin{proof}
	We prove the implication
	\begin{equation}\label{eq:longintervall}
		0 < m \leq \sum_{u=0}^{l}\binom{a - \left\lceil \frac{u(a+1)}{k}\right\rceil}{k-u} \implies \frac{F_k(m)}{m} \geq \frac{k}{a-k+1}
	\end{equation}
	by induction on $l$.
	
	For $l = 0$ it is just Lemma~\ref{lem:sperner}. Let us now suppose that (\ref{eq:longintervall}) is true for some $l$ and also that $\sum_{u=0}^{l}\binom{a - \left\lceil \frac{u(a+1)}{k}\right\rceil}{k-u} < m \leq \sum_{u=0}^{l+1}\binom{a - \left\lceil \frac{u(a+1)}{k}\right\rceil}{k-u}$. We can write $m$ as $m = \sum_{u= 0}^{l}\binom{a - \left\lceil \frac{u(a+1)}{k}\right\rceil}{k-u} + x$, where $0 < x \leq \binom{a - \left\lceil \frac{(l+1)(a+1)}{k}\right\rceil}{k-l-1}$. Let the $(k-l-1)$-binomial representation of $x$ be $x = \sum_{s = l+1}^{k-t}\binom{x_{k-s}}{k-s}$.
	\begin{claim}\label{claim:k-binom}
		The $k$-binomial representation of $m$ is
		\[m = \sum_{u= 0}^{l}\binom{a - \left\lceil \frac{u(a+1)}{k}\right\rceil}{k-u} + \sum_{s = l+1}^{k-t}\binom{x_{k-s}}{k-s}.\]
	\end{claim}
	\begin{proof}
		Since the above sum equals to $m$, the only thing we need to show is $a > a - \left\lceil\frac{a+1}{k}\right\rceil > \ldots > a - \left\lceil \frac{l(a+1)}{k}\right\rceil > x_{k-l-1} > \ldots > x_t$. The only interesting inequality here is $x_{k-l-1} < a - \left\lceil \frac{l(a+1)}{k}\right\rceil$, which immediately follows from $x \leq \binom{a - \left\lceil \frac{(l+1)(a+1)}{k}\right\rceil}{k-l-1}$.
	\end{proof}
	From the $k$-binomial representation of $m$ it can be seen, that $F_k(m) = F_k(m-x) + F_{k-l-1}(x)$. So
	\[
	\frac{F_k(m)}{m} =
	\frac{F_k(m-x) + F_{k-l-1}(x)}{m-x +x} = (1 - \lambda)\frac{F_k(m-x)}{m-x} + \lambda\frac{F_{k-l-1}(x)}{x},
	\]
	where $\lambda = \frac{x}{m} \in (0,1)$.
	
	The inequality $\frac{F_k(m-x)}{m-x} \geq \frac{k}{a-k+1}$ follows from the induction hypothesis. For $\frac{F_{k-l-1}(x)}{x}$, we can use Lemma~\ref{lem:sperner}, which says that since $x \leq \binom{a - \left\lceil \frac{(l+1)(a+1)}{k}\right\rceil}{k-l-1}$, the inequality
	\[
	\frac{F_{k-l-1}(x)}{x} \geq  
	\frac{k-l-1}{a - \left\lceil\frac{(l+1)(a+1)}{k}\right\rceil - k + l + 2} \geq
	\frac{k-l-1}{a - \frac{(l+1)(a+1)}{k} - k+l+2} = \frac{k}{a-k+1}
	\]
	holds.
\end{proof}

In a similar manner one can prove other inequalities for $\frac{F_k(m)}{m}$. We only need one more.

\begin{lemma}\label{lem:shortintervall}
	Let $a \geq k$ be positive integers. If
	\[
	\binom{a+1}{k} - \frac{k-1}{k}(a-k+2) \leq d \leq \binom{a+1}{k},
	\]
	then for every $m \leq d$
	\[
	\frac{F_k(m)}{m} \geq \frac{F_k(d)}{d} = \frac{\binom{a+1}{k-1}}{d}.
	\]
\end{lemma}

\begin{proof}
	If $d = \binom{a+1}{k}$, the statement is just Lemma~\ref{lem:sperner}. Otherwise, the $k$-binomial representation of $d$ is $d = \binom{a}{k} + \binom{a-1}{k-1} + \ldots + \binom{a-k+2}{2} + \binom{l}{1}$ for some $\frac{a-k+2}{k} \leq l < a-k+2$, hence $F_k(d) = \binom{a}{k-1} + \binom{a-1}{k-2} + \ldots + \binom{a-k+2}{1} + \binom{l}{0} = \binom{a+1}{k-1}$.
	
	Let $d_0 = \binom{a}{k} + \binom{a-1}{k-1} + \ldots + \binom{a-k+2}{2}$. We prove the inequality $\frac{F_k(m)}{m} \geq \frac{F_k(d)}{d}$ in three cases $m > d_0$, $m = d_0$ and $m < d_0$ separately.
	
	If $m > d_0$, then $m = \binom{a}{k} + \binom{a-1}{k-1} + \ldots + \binom{a-k+2}{2} + \binom{j}{1}$ for some $j \leq d$, hence
	
	\begin{align*}
	\frac{F_k(m)}{m} &=
	\frac{\binom{a}{k-1} + \ldots + \binom{a-k+2}{1} + \binom{j}{0}}{\binom{a}{k} + \ldots + \binom{a-k+2}{2} + \binom{j}{1}} =
	\frac{F_k(d_0) + 1}{d_0 + j} \geq \\
	\geq \frac{F_k(d_0) + 1}{d_0 + l} &= 
	\frac{\binom{a}{k-1} + \ldots + \binom{a-k+2}{1} + \binom{j}{0}}{\binom{a}{k} + \ldots + \binom{a-k+2}{2} + \binom{l}{1}} = 
	\frac{F_k(d)}{d}.
	\end{align*}
	
	If $m = d_0$, then
	\begin{align*}
	\frac{F_k(m)}{m} =
	\frac{\binom{a}{k-1} + \ldots + \binom{a-k+2}{1}}{\binom{a}{k} + \ldots + \binom{a-k+2}{2}}&=
	\frac{\binom{a+1}{k-1} - 1}{\binom{a+1}{k} - (a-k+2)} > \\
	> \frac{\binom{a+1}{k-1} - 1}{\frac{a-k+2}{k}\binom{a+1}{k-1} - \frac{a-k+2}{k}} &=
	\frac{k}{a-k+2} \geq
	\frac{1}{l},
	\end{align*}
	therefor
	\[
	\frac{F_k(m)}{m} =
	\frac{\binom{a}{k-1} + \ldots + \binom{a-k+2}{1}}{\binom{a}{k} + \ldots + \binom{a-k+2}{2}} >
	\frac{\binom{a}{k-1} + \ldots + \binom{a-k+2}{1} + \binom{l}{0}}{\binom{a}{k} + \ldots + \binom{a-k+2}{2} + \binom{l}{1}} =
	\frac{F_k(d)}{d}.
	\]
	
	If $m < d_0$, then let the $k$-binomial representation of $m$ be $m = \binom{a_k}{k} + \ldots + \binom{a_1}{1}$. For every $0 \leq i \leq k$ the inequality $a_i \leq a - k + i$ holds, so for every $0 \leq i \leq k-1$
	\begin{equation}\label{eq:aibecsles}
	\frac{\binom{a_i}{i-1}}{\binom{a_i}{i}} =
	\frac{i}{a_i - i + 1} \geq
	\frac{i}{a - k + 1} =
	\frac{\binom{a-k + i}{i-1}}{\binom{a-k+i}{i}},
	\end{equation}
	furthermore, according to Lemma~\ref{lem:sperner}, the inequality $\binom{a_2}{2} + \binom{a_1}{1} \leq \binom{a-k+2}{2}$ implies
	\begin{equation}\label{eq:a2becsles}
	\frac{\binom{a_2}{1} + \binom{a_1}{0}}{\binom{a_2}{2} + \binom{a_1}{1}} \geq \frac{\binom{a-k+2}{1}}{\binom{a-k+2}{2}}.
	\end{equation}
	Putting (\ref{eq:aibecsles}), (\ref{eq:a2becsles}) and the $m = d_0$ case together we get
	\begin{align*}
	\frac{F_k(m)}{m} = \frac{\binom{a_k}{k-1} + \ldots + \binom{a_1}{0}}{\binom{a_k}{k} + \ldots + \binom{a_1}{1}} \geq \frac{\binom{a}{k-1} + \ldots + \binom{a-k+2}{1}}{\binom{a}{k} + \ldots + \binom{a-k+2}{2}} =
	\frac{F_k(d_0)}{d_0} \geq
	\frac{F_k(d)}{d},
	\end{align*}
	finishing the proof of Lemma~\ref{lem:shortintervall}.
\end{proof}

\section{Shadow ratio of hypergraph with higher maximum degree}\label{sec:largedegree}

We prove Theorem~\ref{thm:lovaszdegree}, \ref{thm:longintervall} and \ref{thm:shortintervall} by connecting the case of bounded degree hypergraphs to the case of bounded size hypergraphs.

If $v$ is a vertex of $\Hc$, then let $\Hc_v = \{H \in \Hc: v \in H\}$ and $\Hc_{-v} = \{H \setminus \{v\}: H \in \Hc_v\}$. If $\Hc$ is $k$-uniform, $H_{-v}$ is a $(k-1)$-uniform hypergraph, and its shadow ratio is closely related to the shadow ratio of $\Hc$, as stated in the following

\begin{lemma}\label{lem:key}
	If there exists an $\alpha \in \R$ such that for a $k$-uniform hypergraph $\Hc$ the inequality
	\[
	\frac{|\sigma(\Hc_{-v})|}{|\Hc_{-v}|} \geq \alpha
	\]
	holds for every vertex $v$ of $\Hc$, then
	\[
	\sm \geq \frac{k\alpha}{k-1}.
	\]
\end{lemma}

\begin{proof}[Proof of Lemma~\ref{lem:key}]
	By double counting the vertices of $\Hc$ and the hyperedges containing them, we get $k|\Hc| = \sum_v |\Hc_v|$. The same argument for the shadow shows $(k-1)|\sigma(\Hc)| = \sum_v|\sigma(\Hc)_v|$. Here $\sigma(\Hc)_v$ is the set of those elements of the shadow, which contain $v$. The equality $|\Hc_{-v}| = |\Hc_v|$ holds for every $v$, as the bijection $H \mapsto H \cup \{v\}$ shows. Similarly, we can get $|\sigma(\Hc_{-v})| = |\sigma(\Hc)_v|$ for every $v$. The following line of inequalities prove the lemma:
	\[
	\frac{(k-1)|\sigma(\Hc)|}{k|\Hc|} =
	\frac{\sum_v|\sigma(\Hc)_v|}{\sum_v|\Hc_v|} \geq
	\min_v\frac{|\sigma(\Hc)_v|}{|\Hc_v|} =
	\min_v\frac{|\sigma(\Hc_{-v})|}{|\Hc_{-v}|} \geq
	\alpha.
	\]
\end{proof}

Note that $|\Hc_{-v}| = |\Hc_v| \leq \Delta(\Hc)$. We can combine Lemma~\ref{lem:key} with the results about the shadow ratio of bounded size hypergraphs from Section~\ref{sec:sizebound} to get the theorems about bounded degree hypergraphs.

\begin{proof}[Proof of Theorem~\ref{thm:lovaszdegree}]
	According to the degree bound, for every vertex $v$ of $\Hc$ the inequality
	\[
	|\Hc_{-v}| \leq \binom{x}{k-1}
	\]
	holds.
	
	We can apply Corollary~\ref{cor:lovaszshadowratio}, which guarantees
	\[
	\frac{|\sigma(\Hc_{-v})|}{|\Hc_{-v}|} \geq \frac{k-1}{x-k+2}.
	\]
	From that, Lemma~\ref{lem:key} implies
	\[
	\sm \geq \frac{k}{x-k+2}.
	\]
\end{proof}

\begin{proof}[Proof of Theorem~\ref{thm:longintervall}]
	According to the degree bound, for every vertex $v$ of $\Hc$ the inequality
	\[
	|\Hc_{-v}| \leq \binom{t}{k-1} + \binom{t -\left\lceil \frac{(t+1)}{k-1}\right\rceil}{k-2} + \binom{t -\left\lceil \frac{2(t+1)}{k-1}\right\rceil}{k-3} +
	\ldots +
	\binom{t - \left\lceil \frac{(k-2)(t+1)}{k-1}\right\rceil}{1}
	\]
	holds.
	
	We can apply Lemma~\ref{lem:longintervall}, which guarantees
	\[
	\frac{|\sigma(\Hc_{-v})|}{|\Hc_{-v}|} \geq \frac{k-1}{t-k+2}.
	\]
	From that, Lemma~\ref{lem:key} implies
	\[
	\sm \geq \frac{k}{t-k+2}.
	\]
\end{proof}

\begin{proof}[Proof of Theorem~\ref{thm:shortintervall}]
	There is a $d$ such that $\binom{t+1}{k-1} - \frac{k-2}{k-1}(t-k+3) \leq d < \binom{t+1}{k-1}$ and for every vertex $v$ of $\Hc$ the inequality
	\[|\Hc_{-v}| \leq d\]
	holds.
	
	We can apply Lemma~\ref{lem:shortintervall}, which guarantees
	\[
	\frac{|\sigma(\Hc_{-v})|}{|\Hc_{-v}|} \geq \frac{\binom{t+1}{k-2}}{d}.
	\]
	
	From that, Lemma~\ref{lem:key} implies
	\[
	\sm \geq \frac{k\binom{t+1}{k-2}}{(k-1)d}.
	\]
\end{proof}

\section{Shadow of hypergraphs with very low maximum  degree}\label{sec:smalldegree}

\begin{proof}[Proof of Theorem~\ref{thm:lowdegree}]
	First, we show, that if a $k$-uniform hypergraph $\Hc$ has degree bound $d < k$, then its hyperedges can be partitioned into subhypergraphs of size at most $d$ with the subhypergraphs having disjoint shadows.
	
	\begin{definition}
		The subhypergraph $\Hc'$ of a $k$-uniform hypergraph $\Hc$ is connected if for every $H, F \in \Hc'$ there is a sequence of hyperedges $H = H_1, H_2, \ldots, H_t = F$ of $\Hc'$ such that for every $1 \leq i \leq t-1$, $|H_i \cap H_{i+1}| = k-1$.
		
		The maximal connected subhypergraphs of $\Hc$ are called the connected components of $\Hc$.
	\end{definition}
	
	If $\Hc'$ and $\Hc''$ are two different connected components of $\Hc$, then their shadows are disjoint. First, we consider only connected hypergraphs. Their size is at most $d$, as stated in the following
	
	\begin{lemma}
		If $d < k$ and $\Hc$ is a connected $k$-uniform hypergraph with degree bound $d$, then $|\Hc| \leq d$.
	\end{lemma}
	\begin{proof}
		Let $\Hc = \{H_1, \ldots, H_t\}$ and let the hyperedges be ordered in such a way, that for every $H_i$ there is an $H_j$ with $j<i$ and $|H_i \cap H_j| = k-1$. We prove the inequality
		\begin{equation}\label{eq:largecap}
		\forall i \in [t]: |\cap_{j=1}^{i}H_j| \geq k - i + 1
		\end{equation}

		by induction on $i$. For $i = 1$, $|H_1| = k$ is trivially true. Let us suppose that $|\cap_{j=1}^{i-1}H_j| \geq k-i+2$. $H_i$ has a neighbour $H_s$ among $H_1, \ldots, H_{i-1}$. Since there is exactly one element contained in $H_s$ and not contained in $H_i$, intersecting with $H_i$ can not decrease the size of the intersection with more than one element, and so $|\cap_{j=1}^{i}H_j| \geq k - i + 1$.
		
		Now suppose for contradiction that $t \geq d+1$ and use (\ref{eq:largecap}) to conclude, that $|\cap_{j=1}^{d+1}H_j| \geq k - d \geq 1$. Since $\cap_{j=1}^{d+1}H_j \neq \emptyset$, there is an element in $\Hc$ with degree more than $d$, a contradicton.
	\end{proof}
	
	Turning back to the proof of Theorem~\ref{thm:lowdegree}, first consider the case when $\Hc$ is a $k$-uniform connected hypergraph with degree bound $d < k$ and size $t \leq d$. We show, that $|\sigma(\Hc)| \geq kt - \binom{t}{2}$. Let $\Hc = \{H_1, \ldots, H_t\}$ and call the elements of $\sigma(\Hc)$ \textit{reduced edges}. $H_1$ contains $k$ reduced edges. $H_2$ contains at least $k-1$ reduced edges which are not contained in $H_1$. Similarly, for any $i$, $H_i$ contains at least $k-i+1$ reduced edges, which are not contained in $H_1, \ldots, H_{i-1}$. So $|\sigma(\Hc)| \geq \sum_{i = 1}^t k - i + 1 = kt - \binom{t}{2}$. Equality is achieved, if $\Hc$ has exactly $k+1$ vertices.
	
	Finally, if $\Hc$ is a $k$-uniform hypergraph with degree bound $d<k$ and arbitrary size $t$, then it is a disjoint union of subhypergraphs $\Hc_1, \ldots, \Hc_{q+1}$ each of size at most $d$ and with disjoint shadows. Note, that the vertex sets of the subhypergraphs may be intersecting, but their shadows are disjoint. Let the sizes of the subhypergraphs be $t_1, \ldots, t_{q+1}$. We claim that if there are $i \neq j$ with $0 < t_i \leq t_j < d$, then $\Hc$ is not extremal. Indeed, if this is the case, then $|\sigma(\Hc_i)| + |\sigma(\Hc_j)| \geq k(t_1 + t_2) - \binom{t_1}{2} - \binom{t_2}{2}$. But if we replace this two subhypergraphs with $\Hc'_i$ and $\Hc'_j$, subhypergraphs on separate $k+1$ vertices and size $t_i - 1$ and $t_j + 1$, then $|\sigma(\Hc'_i)| + |\sigma(\Hc'_j)| = k(t_1 + t_2) - \binom{t_1 - 1}{2} - \binom{t_2 + 1}{2} < k(t_1 + t_2) - \binom{t_1}{2} - \binom{t_2}{2}$.
	
	As a conclusion, extremal $k$-uniform hypergraphs with degree bound $d<k$ and size $qd + r$ are those, which can be decomposed into $q+1$ connected hypergraphs, each on $k+1$ vertices and $q$ of them having size $d$, one of them size $r$.
	
\end{proof}

\section{Nonregular initial segments with low shadow ratio}\label{sec:example}

The Katona-Kruskal Shadow Theorem can be stated in the following way: a $k$-uniform hypergraph with $m$ hyperedges has shadow size at least that of the size of the shadow of the $k$-uniform \textit{initial segment} of size $m$. Initial segments are in a way a generalization of cliques and can be defined with the help of colexicographic ordering, where if given two sets $A, B \subset [n]$ we say that $A > B$ if and only if the largest element of the symmetric difference of $A$ and $B$ are in $A$. A $k$-uniform \textit{initial segment} of size $m$ is then just the first $m$ subsets of size $k$ in the colexicographic ordering. These constructions are easier to describe when the $k$-binomial representation of $m$ contains only a few terms, but all the initial segments have a pretty low shadow ratio and some of them might even be extremal for a given degree bound. The computation and comparison of the shadow ratios of initial segments are somewhat obscure, so we only do it for one more class of initial segments, to prove Proposition~\ref{prop:example}.

\begin{proof}[Proof of Proposition~\ref{prop:example}]
	We need a hypergraph with degree bound $\binom{t}{k-1} + \binom{t+1 - \left\lceil \frac{t+2}{k} \right\rceil}{k-2}$ and lower shadow ratio than the shadow ratio of $\binom{[t+1]}{k}$. The initial segment of size $\binom{t+1}{k} + \binom{t+2-\left\lceil\frac{t+2}{k}\right\rceil}{k-1}$ is such a hypergraph, as will be shown in the following. Note that this hypergraph is just $\Hc = \binom{[t+1]}{k} \cup \left \{E \cup \{t+2\}: E \in \binom{\left [t + 2 - \left\lceil \frac{t+2}{k} \right\rceil\right ]}{k-1}\right \}$.
	
	The vertices of $\Hc$ can be categorised into three classes according to their degree. Firstly, the degree of $t+2$ is $\binom{t + 1 - \left\lceil \frac{t+2}{k} \right\rceil}{k-1}$. Secondly, if $t + 2 - \left\lceil \frac{t+2}{k} \right\rceil < v < t+2$, then the degree fo $v$ is $\binom{t}{k-1}$. Thirdly, if $v \leq t + 2 - \left\lceil \frac{t+2}{k} \right\rceil$, then the degree of $v$ is  $ \binom{t}{k-1} + \binom{t+1 - \left\lceil \frac{t+2}{k} \right\rceil}{k-2}$. Hence, all the degrees are at most $\binom{t}{k-1} + \binom{t+1 - \left\lceil \frac{t+2}{k} \right\rceil}{k-2}$.
	
	What will be the size of $\sigma(\Hc)$? There are $\binom{t+2-\left\lceil\frac{t+2}{k}\right\rceil}{k-2}$ elements of $\sigma(\Hc)$ containing the last vertex, and $\binom{t+1}{k-1}$ elements of $\sigma(\Hc)$ not conaining it, so $|\sigma(\Hc)| = \binom{t+1}{k-1} + \binom{t+2-\left\lceil\frac{t+2}{k}\right\rceil}{k-2}$.
	
	The shadow ratio of $\Hc$ can be expressed as
	
	\begin{equation*}
	\sm =
	\frac{\binom{t+1}{k-1} + \binom{t + 2 - \left\lceil \frac{t+2}{k} \right\rceil}{k-2}}{\binom{t+1}{k} + \binom{t + 2 - \left\lceil \frac{t+2}{k} \right\rceil}{k-1}} = 
	\lambda\frac{\binom{t+1}{k-1}}{\binom{t+1}{k}} + (1-\lambda)\frac{\binom{t + 2 - \left\lceil \frac{t+2}{k} \right\rceil}{k-2}}{\binom{t + 2- \left\lceil \frac{t+2}{k} \right\rceil}{k-1}},
	\end{equation*}
	
	where $\lambda \in (0,1)$, hence it is a convex combination. The first term is exactly $\frac{k}{t-k+2}$, but for the second term
	\[
	\frac{\binom{t + 2 - \left\lceil \frac{t+2}{k} \right\rceil}{k-2}}{\binom{t + 2 - \left\lceil \frac{t+2}{k} \right\rceil}{k-1}} = 
	\frac{k-1}{t+4 - k - \left\lceil \frac{t+2}{k} \right\rceil} <
	\frac{k-1}{t+3 - k - \frac{t+2}{k}} =
	\frac{k}{t-k+2}
	\]
	holds, therefore
	\[
	\sm < \frac{k}{t-k+2},\]
	proving Proposition~\ref{prop:example}.
\end{proof} 

\section{Concluding remarks}\label{sec:conclusion}

We have determined a sharp lower bound on the shadow ratio of $k$-uniform hypergraphs with degree bound $d$ if $\binom{t}{k-1} \leq d \leq \binom{t}{k-1}
+ \binom{t -\left\lceil \frac{(t+1)}{k-1}\right\rceil}{k-2}
+ \ldots 
+ \binom{t - \left\lceil \frac{(k-2)(t+1)}{k-1}\right\rceil}{1}$ for some $t$ and if $\binom{t+1}{k-1} 
- \frac{k-2}{k-1}(t-k+3) 
\leq d 
< \binom{t+1}{k-1}$ and $k|t+1$ for some $t$. An obvious open question is to determine a sharp lower bound on the shadow ratio for other values of $d$, most importantly if $\binom{t}{k-1}
+ \binom{t -\left\lceil \frac{(t+1)}{k-1}\right\rceil}{k-2}
+ \ldots 
+ \binom{t - \left\lceil \frac{(k-2)(t+1)}{k-1}\right\rceil}{1} < d < \binom{t+2}{k-1} 
- \frac{k-2}{k-1}(t-k+4)$. All the extremal hypergraphs we have seen so far could have been obtained by deleting some (sometimes zero) disjoint maximal matchings from an initial segment.

\begin{question}
	Is it true, that for any $k$ and $d$, among $k$-uniform hypergraphs with degree bound $d$, there is an extremal hypergraph with minimal shadow ratio which can be obtained by deleting disjoint maximal matchings from an initial segment?
\end{question}

In all the solved cases, the extremal hypergraph were regular, because Lemma~\ref{lem:key} can not guarantee extremality if the hypergraph is nonhomogene in term of the shadow ratios of the local subhypergraphs around the vertices.

A harder question is: how can we bound the size of the shadow as a function of $k$, $d$ and $|\Hc|$? Can we find functions $F_k^d$ for which $|\sigma(\Hc)| \geq F_k^d(|\Hc|)$ holds sharply, that is the inequality holds for any $k$ uniform hypergraph with degree bound $d$, and for any size $m$ there is a hypergraph $\Hc$ with $|\Hc| = m$ and $|\sigma(\Hc)| = F_k^d(m)$? We have seen in Theorem~\ref{thm:lowdegree}, that if $d<k$ and $m = qd + r$ with $r < d$ and $q$ integers, then $F_k^d(m) = \left(qd + r\right)k - q\binom{d}{2} - \binom{r}{2}$. What can we say about the case $k \leq d$?

\begin{question}
	Determine functions $F_k^d$ for $k\leq d$ such that $|\sigma(\Hc)| \geq F_k^d(|\Hc|)$ sharply holds for $k$-uniform hypergraphs with degree bound $d$.
\end{question}

If $d$ is large enough compared to $k$ and $m$, then $F_k^d(m)$ equals to $F_k(m)$, the shadow function from the Kruskal-Katona Theorem. Hence an answer to this question would generalize Kruskal-Katona.

\invisiblesection{Acknowledgement}
\subsection*{Acknowledgement}

The author would like to thank Gyula O. H. Katona for providing the problem and for the many valuable discussions and help during the research.

\bibliographystyle{alpha}
\bibliography{biblio_shadow}

\begin{thebibliography}{Kru63}

\bibitem[FF84]{frankl1984hypergraphs}
Peter Frankl and Zolt{\'a}n F{\"u}redi.
\newblock On hypergraphs without two edges intersecting in a given number of
  vertices.
\newblock {\em Journal of Combinatorial Theory, Series A}, 36(2):230--236,
  1984.

\bibitem[Kat64]{katona1964intersection}
Gyula Katona.
\newblock Intersection theorems for systems of finite sets.
\newblock {\em Acta Mathematica Academiae Scientiarum Hungaricae},
  15(3-4):329--337, 1964.

\bibitem[Kat68]{katona1968theorem}
Gyula~O.H. Katona.
\newblock A theorem of finite sets.
\newblock {\em Theory of Graphs}, pages 187--207, 1968.

\bibitem[Kru63]{kruskal1963simplices}
Joseph Kruskal.
\newblock The number of simplicies in a complex.
\newblock pages 251--278. Univ. of California Press, 1963.

\bibitem[Lec93]{leck1993minimum}
Uwe Leck.
\newblock {\em On the minimum size of the shadow of set systems with a SDR}.
\newblock Freie Universit{\"a}t Berlin. Fachbereich Mathematik, 1993.

\bibitem[Lov07]{lovasz2007combinatorial}
L{\'a}szl{\'o} Lov{\'a}sz.
\newblock {\em Combinatorial problems and exercises}, volume 361.
\newblock American Mathematical Soc., 2007.

\bibitem[Spe28]{sperner1928satz}
Emanuel Sperner.
\newblock Ein satz {\"u}ber untermengen einer endlichen menge.
\newblock {\em Mathematische Zeitschrift}, 27(1):544--548, 1928.

\end{thebibliography}

\end{document}